\documentclass{article}

\usepackage{amsthm}
\usepackage[english]{babel}
\usepackage{amsmath}
\usepackage{amssymb}
\usepackage[%
]{hyperref}
\usepackage{booktabs}
\usepackage{enumerate}
\usepackage[left = 25mm, top = 40mm, bottom = 35mm, right = 25mm]{geometry}
\usepackage{multicol}
\usepackage{bbm}
\usepackage[]{todonotes}
\usepackage{setspace}
\usepackage{mathtools}
\usepackage {url}

\theoremstyle{definition}
\newtheorem{definition}{Definition}
\newtheorem{Remark}[definition]{Remark}
\newtheorem{Example}[definition]{Example}

\theoremstyle{plain}
\newtheorem{theorem}[definition]{Theorem} 
\newtheorem{lemma}[definition]{Lemma}
\newtheorem{Proposition}[definition]{Proposition}
\newtheorem{Corollary}[definition]{Corollary}

\newtheorem*{Question*}{Question}

\newtheorem*{theorem*}{Theorem}
\newtheorem{theoremintro}{Theorem}

\newcommand{\soc}{\operatorname{soc}}
\newcommand{\Cl}{\operatorname{Cl}}
\newcommand{\J}{J}

\newcommand{\Z}{\mathbb{Z}}
\newcommand{\ord}{\operatorname{ord}}
\newcommand{\Ann}{\operatorname{Ann}}
\newcommand{\N}{\mathbb{N}}

\newcommand{\Syl}{\operatorname{Syl}}

\newcommand{\F}{\mathbb{F}}

\newcommand{\Ker}{\operatorname{Ker}}
\newcommand{\lAnn}{\operatorname{lAnn}}
\newcommand{\rAnn}{\operatorname{rAnn}}

\newcommand{\DCL}[2]{\operatorname{Cl}_{p',#2}(#1)}

\newcommand{\IDCL}[2]{\operatorname{\overline{Cl}}_{p',#2}(#1)}
\newcommand{\SDCL}[2]{\operatorname{Cl}^+_{p',#2}(#1)}

\newcommand{\SIDCL}[2]{\operatorname{\overline{Cl}}^+_{p',#2}(#1)}
\newcommand{\DCLi}[1]{\operatorname{Cl}_{p',#1}}
\newcommand{\IDCLi}[1]{\operatorname{\overline{Cl}}_{p',#1}}
\newcommand{\SIDCLi}[1]{\operatorname{\overline{Cl}}^+_{p',#1}}
\newcommand{\DCLab}{\operatorname{Cl}_{p'}}
\newcommand{\IDCLab}{\operatorname{\overline{Cl}}_{p'}}
\newcommand{\SIDCLab}{\operatorname{\overline{Cl}}^+_{p'}}

\renewcommand{\Im}{\operatorname{Im}}

\numberwithin{definition}{section}
\numberwithin{equation}{section}

\title{Group algebras in which the socle of the center is an ideal\\[0.2cm]}
\author{\textbf{Sofia Brenner}\footnote{Current address: Department of Mathematics, TU Darmstadt, Germany. E-mail address: \texttt{sofia.brenner@tu-darmstadt.de}} \\
	\normalsize\emph{Institute for Mathematics, Friedrich Schiller University Jena, Germany}\\
	\texttt{sofia.bettina.brenner@uni-jena.de}\\[0.4cm]
	\textbf{Burkhard K\"ulshammer} \\
	\normalsize\emph{Institute for Mathematics, Friedrich Schiller University Jena, Germany}\\
	\texttt{kuelshammer@uni-jena.de}}
\date{\vspace{-0.5cm}}

\begin{document}
\maketitle
\begin{abstract}
\noindent
Let $F$ be a field of characteristic $p > 0$. We study the structure of the finite groups $G$ for which the socle of the center of $FG$ is an ideal in $FG$ and classify the finite $p$-groups $G$ with this property. Moreover, we give an explicit description of the finite groups $G$ for which the Reynolds ideal of $FG$ is an ideal in $FG$.
\end{abstract}

\section{Introduction}
Let $F$ be a field and consider the group algebra $FG$ of a finite group $G$ and its center $ZFG$. The question when the Jacobson radical of $ZFG$ is an ideal in $FG$ has been answered by Clarke \cite{CLA69}, Koshitani \cite{KOS78} and K\"ulshammer \cite{KUL20}. We now study the corresponding problem for the socle $\soc(ZFG)$ of $ZFG$ as well as for the Reynolds ideal $R(FG)$ of $FG$. In a prequel to this paper~\cite{BRE221}, we have already given some approaches to these problems for general symmetric algebras. Now, our aim is to analyze the structure of the finite groups~$G$ for which $\soc(ZFG)$ or $R(FG)$ are ideals of $FG$ in a group-theoretic manner. For the Reynolds ideal, we obtain the following characterization:

\begin{theoremintro}\label{theo:a}
Let $F$ be a field of characteristic $p > 0$ and let $G$ be a finite group. Then the Reynolds ideal $R(FG)$ is an ideal in $FG$ if and only if $G'$ is contained in the $p$-core $O_p(G)$ of $G$.
\end{theoremintro}

As a consequence of this result, it follows that if $\soc(ZFG)$ is an ideal in $FG$, one has $G = P \rtimes H$ for a Sylow $p$-subgroup~$P$ of $G$ and an abelian $p'$-group $H$. Based on this decomposition, we derive some fundamental results on the structure of finite groups $G$ for which $\soc(ZFG)$ is an ideal in $FG$. Subsequently, we classify the finite $p$-groups $G$ with this property:

\begin{theoremintro}\label{theo:pgroups}
Let $F$ be a field of characteristic $p > 0$ and let $G$ be a finite $p$-group. Then $\soc(ZFG)$ is an ideal in $FG$ if and only if 
\begin{enumerate}[(i)]
\item $G$ has nilpotency class at most two, that is, $G' \subseteq Z(G)$ holds, or
\item $p = 2$ and $G' \subseteq Y(G) Z(G)$ with $Y(G) = \langle fg^{-1} \colon \{f,g\} \text{ is a conjugacy class of length 2 of }G \rangle.$
\end{enumerate}
In particular, $G$ is metabelian.
\end{theoremintro}

Note that since the $p$-groups of nilpotency class at most two form a large subclass of the finite $p$-groups, the condition that $\soc(ZFG)$ is an ideal in $FG$ is often satisfied. One implication of Theorem \ref{theo:pgroups} generalizes to arbitrary finite groups:

\begin{theoremintro}\label{theo:c}
Let $F$ be a field of characteristic
$p>0$ and let $G$ be a finite group. Suppose that one of the following holds: 
\begin{enumerate}[(i)]
\item $G' \subseteq Z(O_p(G))$, or
\item $p=2$ and $G' \subseteq Y(O_p(G)) Z(O_p(G))$.
\end{enumerate}
Then $\soc(ZFG)$ is an ideal in $FG$.
\end{theoremintro}

The above results are major ingredients for the proof of the main result of this paper, which is a decomposition of $G$ into a central product:
 
\begin{theoremintro}\label{theo:decompositionofp}
Let $F$ be a field of characteristic $p > 0$. Suppose that $G$ is a finite group for which $\soc(ZFG)$ is an ideal in $FG$ and write $G = P \rtimes H$ for a Sylow $p$-subgroup $P$ of $G$ and an abelian $p'$-group $H$ as before. Then $G$ is the central product of the centralizer $C_P(H)$ and the $p$-residual group $O^p(G)$. Moreover, $\soc(ZFC_P(H))$ and $\soc(ZFO^p(G))$ are ideals in $FC_P(H)$ and $FO^p(G)$, respectively. Furthermore, we have 
\[\soc(ZFG) = (Z(P)G')^+ \cdot FG,\] where $(Z(P)G')^+ \in FG$ denotes the sum of the elements in $Z(P)G'$. 
\end{theoremintro}

This statement will allow us to restrict our investigation to the case $P = G'$. A detailed analysis of the structure of finite groups $G$ for which $\soc(ZFG)$ is an ideal in $FG$, based on the above results, will be carried out in a sequel to this paper. 
\bigskip

We proceed as follows: First, we introduce our notation (see Section \ref{sec:notation}) and study the general structure of the finite groups $G$ for which $\soc(ZFG)$ or $R(FG)$ are ideals in $FG$ (see Section~\ref{sec:general}). In Section~\ref{sec:pgroups}, we classify the $p$-groups $G$ for which $\soc(ZFG)$ is an ideal in $FG$ for a field $F$ of characteristic $p > 0$. In Section \ref{sec:decomp}, we derive the decomposition of $G$ given in Theorem~\ref{theo:decompositionofp}.

\section{Notation}\label{sec:notation}
 
Let $G$ be a finite group and $p$ a prime number. As customary, let $G'$, $Z(G)$ and $\Phi(G)$ denote the derived subgroup, the center and the Frattini subgroup of $G$, respectively. For elements $a,b \in G$, we define their commutator as $[a,b] = aba^{-1}b^{-1}$. We write~$[g]$ for the conjugacy class of $g \in G$ and set $\Cl(G)$ to be the set of conjugacy classes of $G$. The nilpotency class of a nilpotent group $G$ will be denoted by $c(G)$. Recall that every $p$-group is nilpotent. For subsets $S$ and $T$ of $G$, let $C_T(S)$ and $N_T(S)$ denote the centralizer and the normalizer of $S$ in $T$, respectively. As customary, let $O_p(G)$, $O_{p'}(G)$ and $O_{p',p}(G)$ be the $p$-core, the $p'$-core and the $p',p$-core of $G$, respectively. By $O^p(G)$ and $O^{p'}(G)$, we denote the $p$-residual subgroup and the $p'$-residual subgroup of $G$, respectively. As customary, let $g_p$ and $g_{p'}$ be the $p$-part and the $p'$-part of an element $g \in G$, respectively. The $p'$-section of $g$ is given by all elements in $G$ whose $p'$-part is conjugate to~$g_{p'}$. We write $G = G_1 * G_2$ if $G$ is the central product of subgroups $G_1$ and $G_2$, that is, we have $G = \langle G_1, G_2 \rangle$ and $[G_1, G_2] = 1$.
\bigskip

For a field $F$ and a finite-dimensional $F$-algebra $A$, we denote by $J(A)$ and $\soc(A)$ its Jacobson radical and (left) socle, the sum of all minimal left ideals of $A$, respectively. Both $J(A)$ and $\soc(A)$ are ideals in $A$. In this paper, an ideal $I$ of $A$ is always meant to be a two-sided ideal, and we denote it by $I \trianglelefteq A$. Additionally, we study the Reynolds ideal $R(A) \coloneqq \soc(A) \cap Z(A)$ of $A$. Furthermore, let $K(A)$ denote the commutator space of $A$, that is, the $F$-subspace of $A$ spanned by all elements of the form $ab-ba$ with $a,b \in A$.
\bigskip 

In the following, we consider the group algebra $FG$ of $G$ over $F$. Recall that $FG$ is a symmetric algebra with symmetrizing linear form
\begin{equation}\label{eq:symmetrizinglinearform}
\lambda \colon FG \to F, \, \sum_{g \in G} a_g g \mapsto a_1.
\end{equation}
For subsets $S$ and $T$ of $FG$, we write $\lAnn_T(S)$ and $\rAnn_T(S)$ for the left and the right annihilator of $S$ in~$T$, respectively, and $\Ann_T(S)$ if both subspaces coincide. For $H \subseteq G$, we set $H^+ \coloneqq \sum_{h \in H} h \in FG$. It is well-known that the elements $C^+$ with $C \in \Cl(G)$ form an $F$-basis of the center $ZFG$ of $FG$. 
\bigskip

In this paper, we mainly study the Jacobson radical $J(ZFG)$ and the socle $\soc(ZFG)$ of the center of $FG$ as well as the Reynolds ideal $R(FG)$. All three spaces are ideals in $ZFG$, but not necessarily in $FG$. Note that $J(ZFG) = J(FG) \cap ZFG$ holds (see \cite[Theorem 1.10.8]{LIN18}) and that by \cite[Theorem 1.10.22]{LIN18}, we have $\soc(ZFG) = \Ann_{ZFG}(J(ZFG))$. Furthermore, observe that $J(ZFG)$, $\soc(ZFG)$ and $R(FG)$ are ideals in~$FG$ if and only if they are closed under multiplication with elements of $FG$ since they are additively closed. 
\bigskip

We recall the definition of the augmentation ideal
$$\omega(FG) = \left\{\sum_{g \in G} a_g g \in FG \colon \sum_{g \in G} a_g = 0\right\}.$$
An $F$-basis of $\omega(FG)$ is given by $\{1-g \colon 1 \neq g \in G\}$. If $F$ is a field of characteristic $p > 0$ and $G$ is a $p$-group, then $J(FG)$ and $\omega(FG)$ coincide (see \cite[Theorem 1.11.1]{LIN18}). For a normal subgroup $N$ of $G$, we consider the canonical projection 
$$\nu_N \colon FG \to F[G/N],\ \sum_{g \in G} a_g g \mapsto \sum_{g \in G} a_g \cdot gN.$$ Its kernel is given by $\omega(FN) \cdot FG = FG \cdot \omega(FN)$ (see \cite[Proposition 1.6.4]{LIN18}). 

\section{General properties}\label{sec:general}
Let $F$ be a field. In this part, we answer the question for which finite groups $G$ the Reynolds ideal $R(FG)$ is an ideal in $FG$. Moreover, we derive structural results on finite groups $G$ for which $\soc(ZFG)$ is an ideal in $FG$. In the next section, these will be applied in order to classify the finite groups of prime power order with this property.
\bigskip

Concerning the choice of the underlying field $F$, we note the following: 

\begin{Remark}
	$\null$
\begin{enumerate}[(i)]
\item Assume that $F$ is of characteristic zero or of positive characteristic not dividing $|G|$. By Maschke's theorem, the group algebra $FG$ is semisimple. In particular, $J(FG) = J(ZFG) = 0$ follows, which yields $R(FG) = \soc(ZFG) = ZFG$. Since $FG$ is unitary, $\soc(ZFG)$ is an ideal of $FG$ if and only if $ZFG = FG$ holds, that is, if and only if $G$ is abelian.
\item Let $F$ be a field of characteristic $p > 0$ and let $G$ be a finite group. Then $\soc(Z\F_p G)$ is an ideal in $\F_p G$ if and only if $\soc(ZFG)$ is an ideal in $FG$. A similar statement holds for the Reynolds ideal.
\end{enumerate}	
\end{Remark}

From now on until the end of this paper, we therefore assume that $F$ is an algebraically closed field of characteristic $p > 0$. 
\bigskip

This section is organized as follows: We first derive a criterion for $\soc(ZFG) \trianglelefteq FG$ (see Section \ref{sec:criterion}) and answer the question when the Reynolds ideal of $FG$ is an ideal in $FG$ (see Section \ref{sec:reynolds}). In Section~\ref{sec:blocks}, we investigate $p$-blocks of $FG$. Subsequently, we find a basis for $J(ZFG)$ (see Section \ref{sec:basisj}) and construct elements in $\soc(ZFG)$ arising from normal $p$-subgroups of $G$ (see Section \ref{sec:npc}). In Section~\ref{sec:casepabelian}, we study the case that $G'$ is contained in the center of a Sylow $p$-subgroup of $G$. We conclude this part by investigating the transition to quotient groups in Section~\ref{sec:quotientgroups} and studying central products in Section~\ref{sec:centralproducts}.

\subsection{\texorpdfstring{Criterion for $\soc(ZFG) \trianglelefteq FG$}{Criterion for soc(ZFG) trianglelefteq FG}}\label{sec:criterion}
Let $G$ be a finite group. In this section, we derive an equivalent criterion for $\soc(ZFG) \trianglelefteq FG$.

\begin{lemma}\label{lemma:fgcdotkfg}
We have $FG \cdot K(FG) = FG \cdot \omega(FG')$.
\end{lemma}

\begin{proof}
As $FG/\omega(FG') \cdot FG$ is isomorphic to the commutative algebra $F[G/G']$, we have $K(FG) \subseteq \omega(FG') \cdot FG$ and hence $K(FG) \cdot FG \subseteq \omega(FG') \cdot FG$ follows. Now let $f \colon FG \to FG/K(FG) \cdot FG$ be the canonical projection map. For all $a,b \in G$, we have $f([a,b]) = f(a)f(b) f(a)^{-1} f(b)^{-1} = 1$ since $FG/K(FG) \cdot FG$ is a commutative algebra. For $g \in G'$, this yields $f(g) = 1$ and hence $f(g-1) = 0$. This shows $\omega(FG') \subseteq \Ker(f) = K(FG) \cdot FG$, which proves the claim.
\end{proof}

\begin{lemma}\label{lemma:socideal}
The socle $\soc(ZFG)$  is an ideal in $FG$ if and only if $\soc(ZFG) \subseteq (G')^+ \cdot FG$ holds.
\end{lemma}

\begin{proof}
By \cite[Lemma 2.1]{KUL20}, we have $\soc(ZFG) \trianglelefteq FG$ if and only if $K(FG) \cdot \soc(ZFG) = 0$ holds, which is equivalent to $FG \cdot K(FG) \cdot \soc(ZFG) = 0$. By Lemma \ref{lemma:fgcdotkfg}, this is equivalent to $FG \cdot \omega(FG') \cdot \soc(ZFG) = 0$, that is, to $\soc(ZFG) \subseteq \rAnn_{FG}(\omega(FG')) = (G')^+ \cdot FG$ (see \cite[Lemma 3.1.2]{PAS77}).
\end{proof}

\subsection{Reynolds ideal}\label{sec:reynolds}
Let $G$ be a finite group. In this section, we answer the question when the Reynolds ideal $R(FG)$ is an ideal in $FG$. Our main result is the following:

\begin{theorem}\label{theo:reynoldsideal}
The following properties are equivalent:
	\begin{enumerate}[(i)]
		\item $R(FG)$ is an ideal of $FG$. 
		\item $G' \subseteq O_p(G).$
		\item $G = P \rtimes H$ with $P \in \Syl_p(G)$ and an abelian $p'$-group $H$. 
	\end{enumerate} 
	In this case, we have $R(FG) = O_p(G)^+ \cdot FG$. 
\end{theorem}

\begin{proof}
Suppose that $R(FG)$ is an ideal in $FG$. Then $FG$ is a basic $F$-algebra by \cite[Lemma 2.2]{BRE221}. Since $F$ is algebraically closed, this implies that $FG/J(FG)$ is commutative. By Lemma~\ref{lemma:fgcdotkfg}, we have $\omega(FG') \cdot FG = K(FG)\cdot FG \subseteq J(FG)$. Thus, for $g \in G'$, the element $g-1$ is nilpotent. Hence there exists $n \in \N$ with $0 = (g-1)^{p^n} = g^{p^n} -1$. This shows that $G'$ is a $p$-group and hence contained in $O_p(G)$.
\bigskip

Now assume $G' \subseteq O_p(G)$ and let $P \in \Syl_p(G)$. Then $G' \subseteq P$ follows, so $P$ is a normal subgroup of $G$ and $G/P$ is abelian. By the Schur-Zassenhaus theorem, $P$ has a complement $H$ in $G$.
Moreover, $H$ is isomorphic to $G/P$ and thus abelian.
\bigskip

Finally suppose that $G = P \rtimes H$ holds, where $P \in \Syl_p(G)$ and $H$ is an abelian $p'$-group. In particular, we have $P = O_p(G)$. We obtain $J(FG) = \omega(FP) \cdot FG$ and $\soc(FG) = \Ann_{FG}(J(FG)) = P^+ \cdot FG
\subseteq (G')^+ \cdot FG \subseteq ZFG$, so that $R(FG) = P^+ \cdot FG$ is an ideal in $FG$.
\end{proof}

This proves Theorem~\ref{theo:a}. Moreover, we obtain the following necessary condition for $\soc(ZFG) \trianglelefteq FG$:

\begin{Corollary}\label{cor:gph}
If $\soc(ZFG)$ is an ideal of $FG$, we have $G = P \rtimes H$ with $P \in \Syl_p(G)$ and an abelian $p'$-group $H$.
\end{Corollary}

\begin{proof}
By \cite[Lemma 1.3]{BRE221}, $\soc(ZFG) \trianglelefteq FG$ implies $R(FG) \trianglelefteq FG$. With this, the claim follows from Theorem~\ref{theo:reynoldsideal}.
\end{proof}

\begin{Remark}\label{rem:trivianeu}\label{rem:structureCG(P)}
Let $G = P \rtimes H$ with $P \in \Syl_p(G)$ and an abelian $p'$-group $H$.
\begin{enumerate}[(i)]
\item By \cite[Theorem 5.3.5]{GOR68}, we have $P = C_P(H) [P,H]$. Due to $[P,H] \subseteq G'$, this yields $G = HP = H C_P(H) [P,H] = H C_P(H) G'$. Note that $[G,H] = [P,H] = [[P,H],H] = [G',H] = [[G',H],H]$ holds by \cite[Theorem 5.3.6]{GOR68} and that this is a normal subgroup of $PH = G$. 
\item We have $O^p(G) = N$ for $N \coloneqq H [G,H]$: Clearly, $N$ is a normal subgroup of $G$. Since $G/N$ is a $p$-group, we have $O^p(G) \subseteq N$. 
On the other hand, $G/O^p(G)$ is a $p$-group, which implies $H \subseteq O^p(G)$ and hence $N \subseteq O^p(G)$ as $O^p(G)$ is a normal subgroup of $G$. In particular, this implies $O^p(G)' \subseteq [G,H]$. On the other hand, we have $[G,H] = [[G',H],H] \subseteq [O^p(G),O^p(G)] = O^p(G)'$ by (i) and hence $O^p(G)' = [G,H] \in \Syl_p(O^p(G))$ follows.
\item Since $O_{p'}(G)$ is contained in the abelian group $H$ and $[P,O_{p'}(G)] \subseteq P \cap O_{p'}(G) = 1$ holds, we have
$O_{p'}(G) \subseteq Z(G)$. Hence \cite[Theorem 6.3.3]{GOR68} implies $C_G(P) \subseteq O_{p'p}(G) = O_{p'}(G) \times P$, and
we conclude that $C_G(P) = O_{p'}(G) \times Z(P)$ holds.
\item Since $R(FG)$ is spanned by the $p'$-section sums of $G$ (see \cite[Equation (39)]{KUL91}), every $p'$-section is of the form $hP$ for some $h \in H$. 
\end{enumerate}
\end{Remark}

\subsection{\texorpdfstring{Blocks and the $p'$-core}{Blocks and the p'-core}}\label{sec:blocks}
Let $G$ be an arbitrary finite group. In this section, we investigate the conditions $\soc(Z(B)) \trianglelefteq B$ and $R(B) \trianglelefteq B$ for a $p$-block $B$ of $FG$. 

\begin{Remark}\label{rem:blocks}
Let $FG = B_1 \oplus \ldots \oplus B_n$ be the decomposition of $FG$ into its $p$-blocks. Then we have $$\soc(ZFG) = \soc(Z(B_1)) \oplus \ldots \oplus \soc(Z(B_n)).$$
In particular, $\soc(ZFG)$ is an ideal in $FG$ if and only if $\soc(Z(B_i)) \trianglelefteq B_i$ holds for all $i \in \{1, \ldots, n\}$, and the analogous statement is true for the Reynolds ideal. Furthermore, it is known that the principal blocks of $FG$ and $F\bar{G}$ are isomorphic for $\bar{G} \coloneqq G/O_{p'}(G)$.
\end{Remark}

For the Reynolds ideal, we obtain the following result: 

\begin{lemma}\label{lemma:reynoldsblock}
The following are equivalent: 
\begin{enumerate}[(i)]
\item There exists a block $B$ of $FG$ for which $R(B) \trianglelefteq B$ holds. 
\item For the principal block $B_0$ of $FG$, we have $R(B_0) \trianglelefteq B_0$. 
\item $G'$ is contained in $O_{p'p}(G)$. 
\end{enumerate}
\end{lemma}

\begin{proof}
Assume that (i) holds. By \cite[Proposition 4.1]{KUL20}, this implies $B \cong B_0$ and hence (ii) holds. Now assume that (ii) holds. 
By \cite[Remarks 2.2 and 3.1]{KUL20}, every simple $B_0$-module is one-dimensional. Since the intersection of the kernels of the simple $B_0$-modules is given by $O_{p'p}(G)$ (see \cite[Theorem 2]{BRA64}), we obtain $G' \subseteq O_{p'p}(G)$. Finally, assume that (iii) holds. Then we have $\bar{G}' \subseteq O_p(\bar{G})$. Theorem~\ref{theo:reynoldsideal} yields $R(F\bar{G}) \trianglelefteq F\bar{G}$, which implies $R(\bar{B}_0) \trianglelefteq \bar{B}_0$ by Remark~\ref{rem:blocks}. Since $B_0$ and $\bar{B}_0$ are isomorphic, we obtain $R(B_0) \trianglelefteq B_0$. 
\end{proof}

Concerning the analogous problem for the socle of the center, we first observe the following:

\begin{lemma}\label{lemma:socblocks}
The following are equivalent: 
\begin{enumerate}[(i)]
	\item There exists a block $B$ of $FG$ for which $\soc(Z(B)) \trianglelefteq B$ holds. 
	\item For the principal block $B_0$ of $FG$, we have $\soc(Z(B_0)) \trianglelefteq B_0$. 
	\item For the principal block $\bar{B}_0$ of $F\bar{G}$, we have $\soc(Z(\bar{B}_0)) \trianglelefteq \bar{B}_0$.
\end{enumerate}
\end{lemma}

\begin{proof}
As in the proof of Lemma~\ref{lemma:reynoldsblock}, the equivalence of (i) and (ii) follows by \cite[Proposition 4.1]{KUL20} and the equivalence of (ii) and (iii) follows from the fact that $B_0$ and $\bar{B}_0$ are isomorphic.
\end{proof}


This has the following important consequence:

\begin{lemma}\label{lemma:opstrich}
We have $\soc(ZFG) \trianglelefteq FG$ if and only if $R(FG) \trianglelefteq FG$ and $\soc(ZF\bar{G}) \trianglelefteq F\bar{G}$ hold. 
\end{lemma}

\begin{proof}
	If $\soc(ZFG)$ is an ideal of $FG$, then $R(FG) \trianglelefteq FG$ holds by \cite[Lemma 1.3]{BRE221} and $\soc(ZF\bar{G})$ is an ideal of $F\bar{G}$ by \cite[Proposition~2.10]{BRE221}. For the latter, note that $F\bar{G} \cong FG/\Ker(\nu_{O_{p'}(G)})$ can be viewed as a quotient algebra of $FG$. Now let $R(FG)$ and $\soc(ZF\bar{G})$ be ideals in $FG$ and $F\bar{G}$, respectively. By Remark~\ref{rem:blocks}, this yields $\soc(Z(\bar{B}_0)) \trianglelefteq \bar{B}_0$ and hence $\soc(Z(B_0)) \trianglelefteq B_0$ (see Lemma~\ref{lemma:socblocks}). Since $R(FG)$ is an ideal in $FG$, all blocks of $FG$ are isomorphic to $B_0$ by \cite[Proposition 4.1]{KUL20}. By Remark~\ref{rem:blocks}, we then obtain $\soc(ZFG) \trianglelefteq FG$.
\end{proof}

\begin{Remark} Assume that $G$ is of the form $G = P \rtimes H$ with $P \in \Syl_p(G)$ and an abelian $p'$-group $H$. Then $\soc(ZFG) \trianglelefteq FG$ is equivalent to $\soc(ZF\bar{G}) \trianglelefteq F\bar{G}$ (see Theorem~\ref{theo:reynoldsideal} and Lemma~\ref{lemma:opstrich}). By going over to the quotient group $G/O_{p'}(G)$, we may therefore restrict our investigation to groups $G$ with $O_{p'}(G) = 1$.
\end{Remark}

\subsection{\texorpdfstring{Basis for $J(ZFG)$}{Basis for J(ZFG)}}\label{sec:basisj}
Let $G = P \rtimes H$ be a finite group with $P \in \Syl_p(G)$ and an abelian $p'$-group $H$ (see Theorem~\ref{theo:reynoldsideal}). The aim of this section is to determine an $F$-basis for $J(ZFG)$. In the given situation, the kernel of the canonical map $\nu_P \colon FG \to F[G/P]$ is given by $J(FG)$ (see \cite[Corollary 1.11.11]{LIN18}). In the following, we distinguish two types of conjugacy classes:

\begin{Remark}\label{rem:conjugacyclassesph}
	Let $C \in \Cl(G)$. We obtain $|\bar{C}| = 1$ for the image  $\bar{C} \in \Cl(G/P)$ of $C$ in $G/P$ since this group is abelian. Now two cases can occur: 
	\begin{itemize}
		\item $|C|$ is divisible by $p$: Then $\nu_P(C^+) = |C| \cdot \bar{C}^+ = 0$ yields $C^+ \in \Ker(\nu_P) \cap ZFG = \J(ZFG)$. 
		\item $|C|$ is not divisible by $p$: In this case, $|P|$ divides $|C_G(g)|$ for any $g \in C$. This yields $P \subseteq C_G(g)$ and hence $C \subseteq C_G(P)$. As customary, we decompose $g = g_{p'} g_p$ into its $p'$-part and $p$-part. Note that $g_{p'} \in  O_{p'}(G) \subseteq Z(G)$ and $g_p \in Z(P)$ hold by Remark~\ref{rem:structureCG(P)}. Due to $g_{p'} \in Z(G)$, we have $C = g_{p'}[g_p]$ and the element $C^+ - |C| \cdot g_{p'}$ is contained in $\Ker(\nu_P) \cap ZFG = J(ZFG)$. \qedhere
	\end{itemize}
\end{Remark}
%
%

\begin{definition}\label{def:basiselements}
For $C \in \Cl(G)$ with $C \not \subseteq O_{p'}(G)$, we set $b_C \coloneqq C^+$ if $p$ divides $|C|$, and $b_C \coloneqq C^+ - |C| \cdot g_{p'}$ otherwise.
\end{definition}

With this, we obtain the following basis for $J(ZFG)$:

\begin{theorem}\label{theo:structjzfg}
An $F$-basis for $J(ZFG)$ is given by $B \coloneqq \left\{b_C \colon C \in \Cl(G), \, C \not \subseteq O_{p'}(G)\right\}$. 
\end{theorem}

\begin{proof}
	By Remark \ref{rem:conjugacyclassesph}, we have $B \subseteq J(ZFG)$. Note that the elements in $B \cup O_{p'}(G)$ form an $F$-basis for~$ZFG$. Since the algebra $F O_{p'}(G)$ is semisimple, $J(ZFG)$ is spanned by~$B$.
\end{proof}

\begin{Remark}\label{rem:hgraded}
The decomposition $FG = \bigoplus_{h \in H} FhP$ gives rise to an $H$-grading of $FG$. Note that the basis of $J(ZFG)$ given in Theorem~\ref{theo:structjzfg} consists of homogeneous elements with respect to this grading. In particular, $J(ZFG)$ is a $H$-graded subspace of $FG$. It follows that $\soc(ZFG) = \Ann_{ZFG}(J(ZFG))$ is a $H$-graded subspace of $FG$ as well, that is, we have
	\[ \soc(ZFG) = \bigoplus_{h \in H} \left(\soc(ZFG) \cap FhP\right). \]
\end{Remark}

\subsection{\texorpdfstring{Elements in $\soc(ZFG)$}{Elements in soc(ZFG)}}\label{sec:npc}
Let $G$ be an arbitrary finite group. In this section, we study elements of $\soc(ZFG)$ which arise from certain normal $p$-subgroups of $G$. Using these, we show that $G'$ has nilpotency class at most two if $\soc(ZFG)$ is an ideal in~$FG$. Moreover, we derive a decomposition of $G$ which will later be used to prove Theorem~\ref{theo:decompositionofp}.

\begin{lemma}\label{lemma:312general}
	Let $N$ be a normal $p$-subgroup of $G$ and set $M \coloneqq \left\{x \in [N,G] \colon x^p \in [N,[N,G]]\right\}.$ For $C \in \Cl(G)$ with $C \not \subseteq C_G(N)$, we have $\nu_M(C^+) = 0$ and hence $M^+ \cdot C^+ = 0$. In particular, this implies $\nu_N(C^+) = 0$ and $N^+ \cdot C^+= 0$.
\end{lemma}

\begin{proof}
	Note that $M$ is a normal subgroup of $G$. Let $R$ be an orbit of the conjugation action of $N$ on $C$ and consider an element $r \in R$. Then $C \not\subseteq C_G(N)$ implies $N \not \subseteq C_G(r),$ which yields $|R| = |N : C_N(r)| \neq 1.$ Set $X \coloneqq \langle N,R \rangle = \langle N,r \rangle.$ 
	\bigskip
	
	First consider the case $[N,G] \subseteq Z(N)$. Then the map $f \colon N \to N,\ n \mapsto [n,r]$ is a group endomorphism with kernel $C_N(r).$ We set $S  \coloneqq \Im(f)$. Then we have $|R| = |N \colon C_N(r)| = |S|$, so in particular, $|S|$ is a nontrivial power of $p.$ Let $\bar{G} \coloneqq G/M$ and set $\bar{g} \coloneqq gM \in \bar{G}$ for $g \in G$ (similarly for subsets of $G$). Note that $\bar{R}$ is an orbit of the conjugation action of $\bar{N}$ on $\bar{C}$. As before, we obtain 
	$|\bar{R}| = |\bar{N} : C_{\bar{N}}(\bar{r})| = |\bar{S}| = |S : S \cap M|.$
	Since $S \subseteq [N,G]$ is a nontrivial $p$-group, $|S \cap M|$ is divisible by $p.$ With this, we obtain 
	$$\nu_M(R^+) = \frac{|R|}{|\bar{R}|} \cdot \bar{R}^+ =  |S \cap M| \cdot \bar{R}^+ = 0.$$

	Now we consider the general case. Let $L \coloneqq [N,[N,G]]$. We set $\tilde{G}\coloneqq G/L$ and write $\tilde{g} \coloneqq g L \in \tilde{G}$ for $g \in G$ (similarly for subsets of $G$). Note that we have $[\tilde{N}, [\tilde{N}, \tilde{G}]] = 1$ and hence $[\tilde{N},\tilde{G}] \subseteq Z(\tilde{N})$. First assume $C_{\tilde{N}}(\tilde{r}) = \tilde{N}$. For any $n \in N,$ one then has $[n,r] \in L$, which implies
	$\nu_{L}(R^+) = |R| \cdot \tilde{r} = 0.$
	Due to $L \subseteq M$, this yields $\nu_M(R^+)  = 0$. 
	Now assume $C_{\tilde{N}}(\tilde{r}) \subsetneq \tilde{N}$. In particular, we have $\tilde{C} \not \subseteq C_{\tilde{G}}(\tilde{N})$. The first part of the proof yields $\nu_{\tilde{M}}(\tilde{R}^+) = 0$, which implies $$\nu_{\tilde{M}}(\nu_{L}(R^+)) = \nu_{\tilde{M}}\left(\frac{|R|}{|\tilde{R}|} \cdot \tilde{R}^+\right) = \frac{|R|}{|\tilde{R}|} \cdot \nu_{\tilde{M}}(\tilde{R}^+) = 0.$$ Due to $\tilde{G}/\tilde{M} = (G/L)/(M/L) \cong G/M$, the map $\nu_{\tilde{M}} \circ \nu_{L}$ can be identified with $\nu_{M}$ and hence $\nu_M(R^+) = 0$ follows. 	Since $R$ was arbitrary, this yields $\nu_M(C^+)  = 0$. In particular, we have $M^+ \cdot C^+ = 0$.
\end{proof}

\begin{Proposition}\label{prop:22}
Let $N$ be a normal $p$-subgroup of $G$ and set $M \coloneqq \left\{x \in [N,G] \colon x^p \in [N,[N,G]]\right\}$ as in Lemma~\ref{lemma:312general}. Moreover, let $K$ be a characteristic subgroup of $C_G(N)$ which satisfies $K^+ \in \soc(ZFC_G(N))$. Then we have $(MK)^+ \in \soc(ZFG)$. In particular, this applies to $K \coloneqq O^{p'}(C_G(N))$.
\end{Proposition}

\begin{proof}
By Lemma~\ref{lemma:312general}, $ZFG$ is the sum of the subspaces $ZFG \cap FC_G(N)$ and $ZFG \cap \Ker(\nu_M)$. Since $\Ker(\nu_M) = \omega(FM) FG = J(FM)FG \subseteq J(FG)$ holds (see \cite[Proposition 1.6.4]{LIN18}), we have $ZFG \cap \Ker(\nu_M) \subseteq J(ZFG)$. Since $ZFG \cap FC_G(N)$ is contained in $ZFC_G(N)$, the space $J(ZFG \cap FC_G(N)) \subseteq J(ZFC_G(N))$ is
	annihilated by $K^+$. This proves that $(MK)^+$ annihilates $J(ZFG)$. Now let $K \coloneqq O^{p'}(C_G(N))$. Since $K^+$ annihilates $J(FC_G(N)) = J(FK) FC_G(N)$ (see \cite[Theorem 1.11.10]{LIN18}), we have $K^+ \in \soc(ZFC_G(N))$ as required.
\end{proof}

Now we return to the assumption that $G$ is of the form $P \rtimes H$ with $P \in \Syl_p(G)$ and an abelian $p'$-group $H$ as in Theorem~\ref{theo:reynoldsideal}.

\begin{lemma}\label{lemma:selfcentralizingsubgroup}\label{lemma:selfcentralizing}
Suppose that $N$ is a normal $p$-subgroup of $G$. Then $(C_P(N)M)^+ \in \soc(ZFG)$ follows, where $M \coloneqq \left\{x \in [N,G] \colon x^p \in [N,[N,G]]\right\}$ is defined as in Lemma~\ref{lemma:312general}. In particular, we have $(C_P(N)N)^+ \in \soc(ZFG)$. If $\soc(ZFG)$ is an ideal in $FG$, then $G' \subseteq C_P(N) M$ follows.
\end{lemma}

\begin{proof}
	Since $C_P(N)$ is a normal Sylow $p$-subgroup of $C_G(N)$, we have $O^{p'}(C_G(N)) = C_P(N)$. Proposition~\ref{prop:22} then yields $(C_P(N)M)^+ \in \soc(ZFG)$. Since $C_P(N)N$ is a union of cosets of $C_P(N)M$, we obtain $(C_P(N)N)^+ \in \soc(ZFG)$. If $\soc(ZFG)$ is an ideal in $FG$, then $G' \subseteq C_P(N)M$ follows by Lemma~\ref{lemma:socideal}.
\end{proof}

The following result will be particularly useful for our derivation on $p$-groups:
\begin{Corollary}\label{cor:zpdginsoc}
	We have $(Z(P)G')^+ \cdot FG \subseteq \soc(ZFG) \subseteq O_p(Z(G))^+ \cdot FG$. 
\end{Corollary}

\begin{proof}
	By Lemma~\ref{lemma:selfcentralizing}, we obtain $(Z(P)M)^+ \in \soc(ZFG)$ for $M = \{x \in [P,G] \colon x^p \in [P,[P,G]]\} \subseteq G'$. In particular, this implies $(Z(P)G')^+ \in \soc(ZFG)$. 
	Since we have $(Z(P)G')^+ \cdot FG \subseteq (G')^+ \cdot FG \subseteq ZFG$, this implies $(Z(P)G')^+ \cdot FG \subseteq \soc(ZFG)$. Now for $z \in O_p(Z(G))$, the element $z-1$ is nilpotent and hence contained in $J(ZFG)$. For $x = \sum_{g \in G} a_g g \in \soc(ZFG)$, this yields $x \cdot (z-1) = 0$, which translates to $a_g = a_{gz}$ for all $g \in G$. Hence $x \in O_p(Z(G))^+ \cdot FG$ follows.
\end{proof}

Observe that the right inclusion in the preceding lemma holds for arbitrary finite groups.
%
The next result is the central ingredient in the proof of Theorem \ref{theo:decompositionofp}:

\begin{Proposition}\label{prop:nilpotencyclassDG}
Suppose that $G' \subseteq C_P(N)N$ holds for every normal $p$-subgroup $N$ of $G$. Then the following hold: 
\begin{enumerate}[(i)]
\item We have $[P,G'] \subseteq Z(G')$. In particular, this implies $G'' \subseteq Z(P)$ and that the nilpotency class of $G'$ is at most two. Moreover, we obtain $\Phi(G') \subseteq Z(G')$.
\item We have $P = C_P(H) \ast [P,H]$ and $G = C_P(H) \ast O^p(G)$.
\end{enumerate}	
\end{Proposition}

\begin{proof}
$\null$
\begin{enumerate}[(i)]
\item Let $D$ be a critical subgroup of $P$ (in the sense of \cite[Theorem 5.3.11]{GOR68}). Then $D$ is normal in $G$, and $Z(D)$ contains $\Phi(D)$, $C_P(D)$ and $[P,D]$. By assumption, we have $G' \subseteq D C_P(D) = D$. Hence we have 
\[[P,G'] \subseteq [P,D] \subseteq Z(D) \subseteq C_G(G'),\]
which implies $[P,G'] \subseteq Z(G')$. With the $3$-subgroups lemma, we obtain $[G'', P] = [[G',G'],P] = 1$, that is, $G'' \subseteq Z(P)$. Furthermore, for $x \in G'$, we have $x \in D$ and hence $x^p \in Z(D) \subseteq C_G(G')$, which implies $x^p \in Z(G')$.
\item By (i), we have $B \coloneqq [C_P(H),[P,H]] \subseteq [P,G'] \subseteq Z(G')$. Furthermore, $B$ is normal in $C_P(H) [P,H] = P$ and $PH = G$. Due to 
\[[C_P(H),G] = [C_P(H), C_P(H)[P,H]H] = [C_P(H),C_P(H)[P,H]] \subseteq C_P(H) B,\] the subgroup $N \coloneqq C_P(H) B$ is normal in $G$. Moreover, we find $[N,H] = [C_P(H)B,H] = [B,H]$. By assumption, we have $G' \subseteq C_P(N) N$. By Remark~\ref{rem:trivianeu}, this yields 
\[[P,H] = [G',H] \subseteq [C_P(N)N,H] \subseteq [N,H] [C_P(N),H],\]
since for $c \in C_P(N)$, $n \in N$ and $h \in H$, we have
$[cn,h] = c[n,h]c^{-1}[c,h] = [n,h] [c,h]$. Hence 
$[P,H] \subseteq [B,H] [C_P(N),H] \subseteq B C_P(N)$ follows, which yields
\[B = [C_P(H),[P,H]] \subseteq [C_P(H), B C_P(N)] = [C_P(H),B] \subseteq [P,B].\]
Hence $B = 1$ follows, which yields $P = C_P(H) \ast [P,H]$. By Remark~\ref{rem:trivianeu}, this implies $G = C_P(H) * H[P,H] = C_P(H) * O^p(G)$. \qedhere 
\end{enumerate}	
\end{proof}

By Lemma~\ref{lemma:selfcentralizing}, the properties given in Proposition~\ref{prop:nilpotencyclassDG} hold whenever $\soc(ZFG)$ is an ideal in $FG$. We conclude this section with a result on $p$-groups, which is an immediate consequence of Lemma~\ref{lemma:selfcentralizing}: 
\begin{lemma}
	If $G$ is a $p$-group satisfying $\soc(ZFG) \trianglelefteq FG$, then $G$ is metabelian.
\end{lemma}

\begin{proof}
	Let $A$ be a maximal abelian normal subgroup of $G$. Since $C_G(A) = A$ holds, Lemma~\ref{lemma:selfcentralizing} yields $G' \subseteq A$. In particular, $G'$ is abelian. 
\end{proof}

\subsection{\texorpdfstring{Special case $G' \subseteq Z(P)$}{Special case G' \subseteq Z(P)}}\label{sec:casepabelian}
Let $G = P \rtimes H$ be a finite group with $P \in \Syl_p(G)$ and an abelian $p'$-group $H$. In this section, we show that $\soc(ZFG)$ is an ideal in $FG$ if $G' \subseteq Z(P)$ holds.

\begin{lemma}\label{lemma:conjstructurePabelian}
	$\null$
	\begin{enumerate}[(i)]
		\item Let $g \in G$ with $g_p \in Z(P)$. Then $[g] = [h] \cdot [g_{p}]$ holds for $h \in H \cap gP$.
		\item For $u \in Z(P)$ and $h \in C_G(H)$, we have $h[u] \subseteq [hu].$ 
		\item Assume $[P,G] \subseteq Z(P)$. Let $h \in C_G(H)$ and write $[h] = U_h h$ with $U_h \coloneqq \{[a,h] \colon a \in G\}$. Then $U_h$ is a normal subgroup of $G$.
	\end{enumerate}
\end{lemma}

\begin{proof}
	$\null$
	\begin{enumerate}[(i)] 
		\item By Remark~\ref{rem:trivianeu}, $gP$ is a $p'$-section of $G$.
		In particular, $[h]$ is the unique $p'$-conjugacy class contained in $gP$ and hence $[g_{p'}] = [h]$ follows. Since $H$ is abelian, we have $g_{p'} = u h u^{-1}$ for some $u \in P$. Due to $g_p \in Z(P)$, this yields $g = u h g_p u^{-1}$ and hence $[g] = [h g_p].$ We may therefore assume $g_{p'} = h$.  
		For $x = p_x h_x$ with $p_x \in P$ and $h_x \in H$, we have $x g x^{-1} = p_x h p_x^{-1} \cdot h_x g_p h_x^{-1}$. This yields $$[g] = \{p_x h p_x^{-1} \colon p_x \in P \} \cdot \{h_x g_p h_x^{-1} \colon h_x \in H\} = [h] \cdot [g_p].$$
		
		\item Let $u' \in [u]$. Due to $u \in Z(P)$, there exists an element $h' \in H$ with $h' u h'^{-1} = u'$ (see Remark~\ref{rem:trivianeu}). Since $h$ and $h'$ commute, we obtain $h u' = h' h u h'^{-1} \in [hu]$. 
		\item We have $U_h = \{[a,h] \colon a \in P\}$. As $[p_1 p_2,h] = [p_1,h]\cdot [p_2,h]$ holds for all $p_1, p_2 \in P$, $U_h$ is a subgroup of $G'$. Since the elements of $P$ centralize $U_h \subseteq [P,G] \subseteq Z(P)$ and conjugation with elements of $H$ permutes the elements $[a,h]$ with $a \in P$, it follows that $U_h$ is normal in $G$. \qedhere
	\end{enumerate}
\end{proof}

\begin{Corollary}\label{cor:PabelianmultwithHsuffices}
Let $g \in G$ with $g_p \in Z(P)$. For $y \in ZFG$ with $y \cdot [g_{p'}]^+ = 0$, we have $y \cdot [g]^+ = 0$. 	
\end{Corollary}

\begin{proof}
The group $P$ acts on $[g]$ by conjugation with orbits of the form $[g_{p'}]u$ with $u \in P$ (see Lemma~\ref{lemma:conjstructurePabelian}). In particular, $[g]$ is a disjoint union of sets of this form. Hence $y \cdot [g_{p'}]^+$ implies $y \cdot [g]^+ = 0$.
\end{proof}	

\begin{lemma}\label{lemma:pabelianaux}
Let $y = \sum_{g \in G} a_g g \in \soc(ZFG)$. For $h \in C_G(H)$ and $u \in Z(P)$, we have $a_{hu} = a_h$. 
\end{lemma}

\begin{proof}
We may assume $u \neq 1$. By Remark~\ref{rem:structureCG(P)}, $m \coloneqq |[u]|$ is not divisible by $p$. Hence we have $b_{[u^{-1}]} = [u^{-1}]^+ - m \cdot 1$ (see Theorem~\ref{theo:structjzfg}) and the coefficient of $h$ in $y \cdot b_{[u^{-1}]} = 0$ is given by
	$$ \sum_{u' \in [u]} a_{hu'} -m a_h = m (a_{hu} - a_{h}),$$
	since the elements in $h[u]$ are conjugate by Lemma \ref{lemma:conjstructurePabelian}\,(ii). Since $p$ does not divide $m$, we obtain $a_{h u} = a_{h}$. 
\end{proof}

\begin{theorem}\label{theo:pabeliansocideal}
If $G = C_G(H)Z(P)$ holds, then $\soc(ZFG) \subseteq Z(P)^+ \cdot FG$ follows. In particular, if we have $G' \subseteq Z(P)$, then $\soc(ZFG)$ is an ideal in $FG$.
\end{theorem}

\begin{proof}
Consider an element $y = \sum_{g \in G} a_g g \in \soc(ZFG)$. Let $g \in G$ and write $g = cz$ with $c \in C_G(H)$ and $z \in Z(P)$. By Lemma~\ref{lemma:pabelianaux}, we have $a_g = a_{cz} = a_c$. Hence $y \in Z(P)^+ \cdot FG$ follows. If additionally $G' \subseteq Z(P)$ holds, then $\soc(ZFG) \subseteq Z(P)^+ \cdot FG \subseteq (G')^+ \cdot FG$ follows, so $\soc(ZFG)$ is an ideal in $FG$ (see Lemma~\ref{lemma:socideal}).
\end{proof}

This proves the first part of Theorem~\ref{theo:c}. The next example shows that the condition $G' \subseteq Z(P)$ is not necessary for $\soc(ZFG) \trianglelefteq FG$.

\begin{Example}\label{ex:172}
	Let $F$ be an algebraically closed field of characteristic $p = 3$ and consider the group $G = \texttt{SmallGroup}(216,86)$ in GAP \cite{GAP4}. We have $G = G' \rtimes H$, where $G'$ is the extraspecial group of order~$27$ and exponent three, and $H \cong C_8$ permutes the nontrivial elements of $G'/G''$ transitively and acts on $G'' = Z(G')$ by inversion. In particular, $G'$ is nonabelian. For $h \in H$, we set $S_h \coloneqq \soc(ZFG) \cap FhG'$. Due to the $H$-grading of $FG$ introduced in Remark~\ref{rem:hgraded}, it suffices to show $S_h = F(hG')^+$ for all $h \in H$. Clearly, we have $(hG')^+ \in S_h$. The derived subgroup $G'$ decomposes into the $G$-conjugacy classes $\{1\}$, $G'' \backslash \{1\}$ and $G' \backslash G''$. For $1 \neq h \in H$, the coset $hG'$ consists of a single conjugacy class for $\ord(h) = 8$ and of two conjugacy classes for $\ord(h) \in \{2,4\}$. In the first case, we directly obtain $S_h = F(hG')^+$. In the latter case, we have $[h]^+ \cdot (G'')^+ = (hG')^+ \neq 0$, which implies $[h]^+ \notin \soc(ZFG)$ since $(G'')^+ \in J(ZFG)$ holds. Since $(hG')^+ - [h]^+ \notin \soc(ZFG)$ holds as well, $S_h = F(hG')^+$ follows. Moreover, this shows $(G'')^+ \notin \soc(ZFG)$ and hence $S_1 = F(G')^+$ follows as well. By Lemma~\ref{lemma:socideal}, $\soc(ZFG)$ is an ideal of $FG$. 
\end{Example}

\subsection{Quotient groups}\label{sec:quotientgroups}
Let $G$ be a finite group of the form $P \rtimes H$ with $P \in \Syl_p(G)$ and an abelian $p'$-group $H$. We fix a normal subgroup $N \trianglelefteq G$ with quotient group $\bar{G} \coloneqq G/N$. Our aim is to study the transition to the group algebra~$F\bar{G}$.
The image of an element $g \in G$ in $\bar
G$ will be denoted by $\bar{g}$ (similarly for subsets of $G$). Note that $\bar{G}$ is of the form $\bar{P} \rtimes \bar{H}$ with $\bar{P} \in \Syl_p(\bar{G})$ and the abelian $p'$-group~$\bar{H}$. In the following, we consider the canonical projection map 
$$\nu_N\colon FG \to F\bar{G},\ \sum_{g \in G} a_g g \to \sum_{g \in G} a_g \cdot gN,$$
together with its adjoint map $\nu_N^\ast \colon F\bar{G} \to FG$, which is defined by requiring $\lambda(\nu_N^\ast(x) y) = \bar{\lambda}(x \nu_N(y))$ for all $x\in F\bar{G}$ and $y \in FG$. Here, $\lambda$ and $\bar{\lambda}$ denote the symmetrizing linear forms of $FG$ and $F\bar{G}$ given in~\eqref{eq:symmetrizinglinearform}, respectively. It is easily verified that $\nu_N^\ast$ is given by
$$\nu_N^\ast \colon F\bar{G} \to FG,\ \sum_{gN \in \bar{G}} a_{gN} \cdot gN \mapsto \sum_{g \in G} a_{gN} \cdot g.$$
Note that $\nu_N^\ast$ is a linear map with image $N^+ \cdot FG$ and that it is injective as $\nu_N$ is surjective. 
%
\begin{Remark}\label{rem:proplambda}
	For $a \in F\bar{G},$ it is easily seen that $a \in (\bar{G}')^+ \cdot F\bar{G}$ is equivalent to $\nu_N^\ast(a) \in (G')^+ \cdot FG$. 
\end{Remark}

If $\soc(ZFG)$ is an ideal in $FG$, then $\Ann_{ZF\bar{G}}(\nu_N(J(ZFG)))$ is an ideal in $F\bar{G}$ by \cite[Proposition 2.10]{BRE221}. For $C \in \Cl(G)$ with $C \not \subseteq O_{p'}(G)$, let $b_C$ denote the associated element of $J(ZFG)$ (see Definition \ref{def:basiselements}) and consider the basis $B \coloneqq \{b_C \colon C \in \Cl(G),\, C \not \subseteq O_{p'}(G)\}$ of $J(ZFG)$ (see Theorem \ref{theo:structjzfg}). Clearly, $\nu_N(J(ZFG))$ is spanned by the images of the elements in $B$. We now derive a more convenient generating set. 
%
%
%
\begin{lemma}\label{lemma:donotlandinopstrich}
Let $C \in \Cl(G)$ be a conjugacy class with $C \not \subseteq O_{p'}(G)$. 
We have $b_C \notin \Ker(\nu_N)$ if and only if $\bar{C} \not \subseteq O_{p'}(\bar{G})$ holds and $k \coloneqq |C|/|\bar{C}|$ is not divisible by $p$. In this case, the basis element $b_{\bar{C}}$ of $J(ZF\bar{G})$ corresponding to $\bar{C} \in \Cl(\bar{G})$ is well-defined and we have $\nu_N(b_C) = k \cdot b_{\bar{C}}$.  

\end{lemma}

\begin{proof}
Observe that $\bar{C}$ is indeed a conjugacy class of $\bar{G}$ and that $\nu_N(C^+) = k \cdot \bar{C}^+$ with $k \coloneqq |C|/|\bar{C}|$ holds. Suppose first that $p$ divides $|C|$, so $b_C = C^+$ holds. Then $\nu_N(b_C) \neq 0$ is equivalent to $k \not\equiv 0 \pmod{p}$, and in this case we have $|\bar{C}| \equiv 0 \pmod{p}$. Since $O_{p'}(\bar{G}) \subseteq Z(\bar{G})$ holds, this implies
$\bar{C} \not\subseteq O_{p'}(\bar{G})$. Moreover, we have $b_{\bar{C}} =
\bar{C}^+$ and thus $\nu_N(b_C) = k \cdot b_{\bar{C}}$.
\bigskip

It remains to consider the case $C \subseteq C_G(P)$. There, we have $\bar{C}
\subseteq C_{\bar{G}}({\bar{P}})$. If $\bar{C} \not\subseteq O_{p'}(\bar{G})$ holds,
then $b_{\bar{C}}$ is defined, and we have $b_C = C^+ - |C| \cdot g_{p'}$ and
$b_{\bar{C}} = \bar{C}^+ - |\bar{C}| \cdot {\bar g}_{p'}$ for $g \in C$. This shows
that $\nu_N(b_C) = k \cdot b_{\bar{C}}$ holds. If, in addition, $k \not\equiv 0 \pmod{p}$, then $\nu_N(b_C) \neq 0$ follows. Suppose conversely that $\nu_N(b_C) \neq 0$ holds. We write $C = g_{p'}D$ for $g_{p'} \in
O_{p'}(G)$ and $D \in \Cl(G)$ with $D \subseteq Z(P)$ (see Remark~\ref{rem:conjugacyclassesph}). Assume that ${\bar C}
\subseteq O_{p'}({\bar G})$ holds. Then we have ${\bar D} = {\bar g}_{p'}^{-1} {\bar C}
\subseteq O_{p'}({\bar G})$ due to ${\bar g}_{p'} \in O_{p'}({\bar G})$. As $D$ consists of $p$-elements, we must have ${\bar D} = \{1\}$, which yields the contradiction $\nu_N(b_C) = \nu_N(g_{p'}D^+ - |D|\cdot g_{p'}) = 0$. This shows that ${\bar C} \not\subseteq O_{p'}({\bar G})$ holds. Hence we have
$\nu_N(b_C) = k \cdot b_{\bar C}$, so that $k \not\equiv 0 \pmod{p}$. 
\end{proof}


\begin{definition}\label{def:abar} \label{def:cln}
		Set 
		$\DCL{G}{N} \coloneqq \{C \in \Cl(G) \colon C \not \subseteq O_{p'}(G) \text{ and } b_C \notin \Ker(\nu_N)\}$ and let $$\SDCL{G}{N} \coloneqq \left\{b_C \colon C \in \DCL{G}{N}\right\}$$ be the set of corresponding basis elements of $J(ZFG)$ (see Definition \ref{def:basiselements}). By $\IDCL{G}{N} \subseteq \Cl(\bar{G}),$ we denote the set of images of the conjugacy classes in $\DCL{G}{N}$ and set $$\SIDCL{G}{N} \coloneqq \left\{b_{\bar{C}} \colon \bar{C} \in \IDCL{G}{N}\right\},$$ where $b_{\bar{C}}$ denotes the basis element of $J(ZF\bar{G})$ corresponding to $\bar{C}$.
\end{definition}

If $N$ is a $p$-group, the $p'$-conjugacy classes of length divisible by $p$ in $\DCL{G}{N}$ can be easily characterized:

\begin{lemma}
Consider a normal $p$-subgroup $N$ of $G$
and let $C \not \subseteq C_G(P)$ be a $p'$-conjugacy class. Then we have $C\in \DCL{G}{N}$ if and only if $C \subseteq C_G(N)$ holds.
\end{lemma} 

\begin{proof}
	If $C \not \subseteq C_G(N)$ holds, we have $\nu_N(b_{C}) = \nu_N(C^+) = 0$ by Lemma~\ref{lemma:312general}, so $C \notin \DCL{G}{N}$. 
	Now let $h \in C
	\subseteq C_G(N)$. Since $h$ is a $p'$-element, \cite[Theorem 5.3.15]{GOR68} implies $C_{G/N}(hN) = C_G(h)N/N = C_G(h)/N$ and hence $|\bar{C}| = |G/N:C_{G/N}(hN)| = |G:C_G(h)| = 
	|C|$. Thus we have $\nu_N(b_C) = \nu_N(C^+) =
	{\bar C}^+ \neq 0$, which yields $C \in \DCL{G}{N}$.
\end{proof}

Now let $N$ again be an arbitrary normal subgroup of $G$. We obtain the following necessary condition for $\soc(ZFG) \trianglelefteq FG$:

\begin{theorem}\label{theo:anndecconjconstantcoeffs}
We have 
$$\Ann_{ZF\bar{G}}(\nu_N(J(ZFG))) = \Ann_{ZF\bar{G}}\bigl(\SIDCL{G}{N}\bigr) \eqqcolon A.$$ If $\soc(ZFG)$ is an ideal of $FG$, we have
$A \subseteq (\bar{G}')^+ \cdot F\bar{G}$.
\end{theorem} 

\begin{proof}
Clearly, the elements $\nu_N(b_C)$ with $C \in \DCL{G}{N}$ span $\nu_N(J(ZFG))$. For $C \in \DCL{G}{N}$ and $y \in F\bar{G}$, we have $y \cdot \nu_N(b_C) = 0$ if and only if $y \cdot b_{\bar{C}} = 0$ holds (see Lemma~\ref{lemma:donotlandinopstrich}). This implies $A = \Ann_{ZF\bar{G}}(\nu_N(J(ZFG)))$. Now assume that $\soc(ZFG)$ is an ideal in $FG$. By \cite[Proposition 2.10]{BRE221}, $A$ is an ideal in $F\bar{G}$, so by \cite[Lemma 2.1]{KUL20}, we have $K(F\bar{G}) \cdot A = 0$. As in the proof of Lemma \ref{lemma:socideal}, this implies $A \subseteq (\bar{G}')^+ \cdot F\bar{G}$.
\end{proof}

As a first application, we give an alternative proof of the following special case of \cite[Proposition 2.10]{BRE221}:

\begin{Corollary}\label{cor:socquotient}
Let $\soc(ZFG)$ be an ideal of $FG$. Then $\soc(ZF\bar{G}) \trianglelefteq F\bar{G}$ holds.
\end{Corollary}

\begin{proof}
Since $\SIDCL{G}{N}$ is a subset of $J(ZF\bar{G}),$ Theorem \ref{theo:anndecconjconstantcoeffs} yields $$\soc(ZF\bar{G}) = \Ann_{ZF\bar{G}}{J(ZF\bar{G})} \subseteq \Ann_{ZF\bar{G}}\bigl(\SIDCL{G}{N}\bigr) \subseteq (\bar{G}')^+ \cdot F\bar{G}$$ and we obtain $\soc(ZF\bar{G}) \trianglelefteq F\bar{G}$ by Lemma \ref{lemma:socideal}. 
\end{proof}

\subsection{Central products}\label{sec:centralproducts}
Let $G$ be a finite group. We consider the question when $\soc(ZFG)$ is an ideal of $FG$ in case that $G = G_1 * G_2$ is a central product of two subgroups $G_1$ and $G_2$. Central products will play an important role throughout our investigation, for instance in the decomposition of $G$ given in Theorem~\ref{theo:decompositionofp}.

\begin{theorem}\label{theo:centralproduct}
Let $G = G_1 * G_2$ be the central product of $G_1$ and $G_2$. Then $\soc(ZFG) \trianglelefteq FG$ is equivalent to $\soc(ZFG_i) \trianglelefteq FG_i$ for $i = 1, 2$. 
\end{theorem}

\begin{proof}
First assume that $\soc(ZFG_i)$ is an ideal in $FG_i$ for $i =1,2$. By \cite[Proposition 1.9]{BRE221}, this implies $$\soc(Z(FG_1 \otimes_F FG_2)) \trianglelefteq FG_1 \otimes_F FG_2.$$ Since $F(G_1 \times G_2) \cong FG_1 \otimes_F FG_2$ holds, this yields $\soc(ZF(G_1 \times G_2)) \trianglelefteq F(G_1 \times G_2)$. The group $G$ is isomorphic to a quotient group of $G_1 \times G_2$, so $\soc(ZFG)$ is an ideal in $FG$ by Corollary~\ref{cor:socquotient}.
\bigskip

Now assume conversely that $\soc(ZFG)$ is an ideal of $FG$. By Corollary~\ref{cor:gph}, $G$ is of the form $P \rtimes H$ with $P \in \Syl_p(G)$ and an abelian $p'$-group $H$. First suppose that $O_{p'}(G) = 1$ holds. Then $Z \coloneqq G_1 \cap G_2 \subseteq Z(G)\subseteq C_G(P) = Z(P)$ is a $p$-group.  We consider the canonical projection $\nu \coloneqq \nu_{G_2} \colon FG \to F[G/G_2]$. By Theorem~\ref{theo:anndecconjconstantcoeffs}, we have 
\begin{equation}\label{eq:anncentralproduct}
\Ann_{ZF[G/G_2]}\bigl(\nu(J(ZFG))\bigr) \subseteq ([G/G_2]')^+ \cdot F[G/G_2].
\end{equation}

Note that there is a canonical isomorphism $G_1/Z \cong G/G_2$. Furthermore, we have $ZFG_1 \subseteq ZFG$ and $\nu(ZFG_1) = \nu(ZFG)$, so also $\nu(J(ZFG_1)) = \nu(J(ZFG))$ holds. Hence we have
$$\Ann_{ZF[G_1/Z]}\bigl(\nu_1 (J(ZFG_1))\bigr) \subseteq ([G_1/Z]')^+ \cdot F[G_1/Z],$$
where $\nu_1 \colon FG_1 \to F[G_1/Z]$ denotes the canonical projection. Let $x_1 \in \soc(ZFG_1)$ and observe that $G_1'$ is a $p$-group.  By Corollary~\ref{cor:zpdginsoc}, we have $x_1 \in Z^+ \cdot FG_1 = \nu_1^\ast(F[G_1/Z])$. Let $y_1 \in F[G_1/Z]$ with $x_1 = \nu_1^\ast(y_1)$. Then \cite[Remark 2.9]{BRE221} yields
$$y_1 \in \Ann_{ZF[G_1/Z]}(\nu_1(J(ZFG_1))) \subseteq ([G_1/Z]')^+ \cdot F[G_1/Z].$$
By Remark \ref{rem:proplambda}, this yields $x_1 \in (G_1')^+ \cdot FG_1$ and hence $\soc(ZFG_1)$ is an ideal in $FG_1$ (see Lemma~\ref{lemma:socideal}). By symmetry, we obtain $\soc(ZFG_2) \trianglelefteq FG_2$. 
\bigskip

Now we consider the general case. For $\bar{G} \coloneqq G/O_{p'}(G)$, we have $\bar{G} = \bar{G}_1 * \bar{G}_2$ with $\bar{G}_i \coloneqq G_i O_{p'}(G)/O_{p'}(G)$ ($i = 1,2$). Note that $\bar{G}_i \cong G_i/O_{p'}(G) \cap G_i \cong G_i/O_{p'}(G_i)$ follows since $O_{p'}(G) \cap G_i = O_{p'}(G_i)$ holds. By the above, we obtain $\soc(ZF\bar{G}_i) \trianglelefteq F\bar{G}_i$. Since $G'$ is a $p$-group, also $G_1'$ and $G_2'$ are $p$-groups. Lemma~\ref{lemma:opstrich} then yields $\soc(ZFG_i) \trianglelefteq FG_i$ for $i = 1,2$.
\end{proof}

\begin{Remark}
For $G \cong G_1 \times G_2$, the statement of Theorem \ref{theo:centralproduct} is a special case of \cite[Proposition~1.9]{BRE221}.
\end{Remark}

\section{Groups of prime power order}\label{sec:pgroups}
Let $F$ be an algebraically closed field of characteristic $p > 0$. In this section, we classify the finite $p$-groups~$G$ for which $\soc(ZFG)$ is an ideal in $FG$ (see Theorem~\ref{theo:pgroups}). Additionally, these results will be generalized to arbitrary finite groups (see Theorem~\ref{theo:c}). First we prove that the property $\soc(ZFG) \trianglelefteq FG$ is preserved under isoclinism (see Section \ref{sec:isoclinism}). Subsequently, we distinguish the cases $p \geq 3$ (see Section \ref{sec:oddcharacteristic}) and $p  = 2$ (see Section \ref{sec:evencharacteristic}).
 
\subsection{Isoclinism}\label{sec:isoclinism}
Let $G$ be a finite $p$-group. The aim of this section is to show that the property $\soc(ZFG) \trianglelefteq FG$ is invariant under isoclinism in the following sense: If $Q$ is a finite $p$-group isoclinic to $G$, then $\soc(ZFQ) \trianglelefteq FQ$ holds precisely if we have $\soc(ZFG) \trianglelefteq FG$. The proof of this statement is based on some observations on the center of $G$ and the transition to the quotient group $\bar{G} \coloneqq G/Z(G)$. 

\begin{lemma}\label{lemma:centerpgroups}
$\null$
\begin{enumerate}[(i)]
\item We have $\soc(ZFG) \subseteq Z(G)^+ \cdot FG$. 
\item $\soc(ZFG)$ is an ideal of $FG$ if and only if $\soc(ZFG) = (Z(G) G')^+ \cdot FG$ holds.
\end{enumerate}
\end{lemma}

\begin{proof}
The first statement follows by Corollary \ref{cor:zpdginsoc}. Now let $\soc(ZFG)$ be an ideal of $FG$. Lemma~\ref{lemma:socideal} then yields $\soc(ZFG) \subseteq (G')^+ \cdot FG$. Together with (i), this implies $\soc(ZFG) \subseteq (Z(G) G')^+ \cdot FG$, and by Corollary~\ref{cor:zpdginsoc}, we obtain equality. Conversely, $(Z(G)G')^+ \cdot FG$ is obviously an ideal in $FG$. \qedhere
\end{proof}

In the given situation, we have
$$\DCLab \coloneqq \DCL{G}{Z(G)} = \{C \in \Cl(G) \colon C \not \subseteq Z(G),\ |C| = |\bar{C}|\}.$$ Note that the length of every conjugacy class in $\DCLab$ is a nontrivial power of $p$. Let $\IDCLab \coloneqq \IDCL{G}{Z(G)}$ be the set of images of the classes in $\DCLab$ and denote by $\SIDCLab \coloneqq \SIDCL{G}{Z(G)}$ the corresponding class sums in~$F\bar{G}$. In this situation, the implication given in Theorem \ref{theo:anndecconjconstantcoeffs} is an equivalence:
%

\begin{lemma}\label{lemma:correspondencecenter}
The socle $\soc(ZFG)$ is an ideal in $FG$ if and only if $\Ann_{ZF \bar{G}}(\SIDCLab) \subseteq (\bar{G}')^+ \cdot F\bar{G}$ holds.
\end{lemma}

\begin{proof}
Consider the map $\nu_{Z(G)}^\ast \colon F\bar{G} \to FG$ introduced in Section \ref{sec:quotientgroups}. Lemma \ref{lemma:centerpgroups} yields $$\soc(ZFG) \subseteq Z(G)^+ \cdot FG = \Im \nu_{Z(G)}^\ast.$$ By \cite[Remark 2.9]{BRE221}, we therefore obtain
$$\soc(ZFG) = \nu^\ast_{Z(G)}\bigl(\Ann_{ZF \bar{G}}(\SIDCLab)\bigr).$$ By Remark \ref{rem:proplambda}, we have $\Ann_{ZF \bar{G}}(\SIDCLab) \subseteq (\bar{G}')^+ \cdot F\bar{G}$ if and only if $\soc(ZFG) \subseteq (G')^+ \cdot FG$ holds, which is equivalent to $\soc(ZFG) \trianglelefteq FG$ by Lemma \ref{lemma:socideal}.
\end{proof}

Now we proceed to the main result of this section. Recall that two finite $p$-groups $G_1$ and $G_2$ are isoclinic if there exist isomorphisms $\varphi \colon G_1' \to G_2'$ and $\beta \colon G_1/Z(G_1) \to G_2/Z(G_2)$ such that whenever $\beta(a_1 Z(G_1)) = a_2 Z(G_2)$ and $\beta(b_1 Z(G_1)) = b_2 Z(G_2)$ hold for $a_1, b_1 \in G_1$ and $a_2, b_2 \in G_2$, then $\varphi([a_1, b_1]) = [a_2, b_2]$ follows. We set $\bar{G}_i \coloneqq G_i/Z(G_i)$ and write $\DCLi{i}$ and $\IDCLi{i}$ to distinguish the sets $\DCLab$ and $\IDCLab$ for $i \in \{1,2\}$. 

\begin{theorem}\label{theo:isoclinic}
Let $G_1$ and $G_2$ be finite isoclinic $p$-groups. Then $\soc(ZFG_1) \trianglelefteq FG_1$ is equivalent to $\soc(ZFG_2) \trianglelefteq FG_2$. 
\end{theorem}

\begin{proof}
Let $\varphi \colon G_1' \to G_2'$ and $\beta \colon \bar{G}_1 \to \bar{G}_2$ be the corresponding isomorphisms. We first show that $\IDCLi{1}$ and $\IDCLi{2}$ are in bijective correspondence under $\beta$. 	Let $C_1 \in \DCLi{1}$ and set $\bar{C}_1$ to be its image in $\bar{G}_1$. Then $\bar{C}_2 \coloneqq \beta(\bar{C}_1)$ is a conjugacy class of $\bar{G}_2$. Consider a preimage $C_2 \in \Cl(G_2)$ of $\bar{C}_2$. Let $x_2 \in C_2$ and assume that $1 \neq [x_2, g_2] \in Z(G_2)$ holds for some $g_2 \in G_2$. Choose elements $x_1 \in C_1$ and $g_1 \in G_1$ with $\beta(x_1 Z(G_1)) = x_2 Z(G_2)$ and $\beta(g_1 Z(G_1)) = g_2 Z(G_2)$. We then obtain $\varphi([x_1, g_1]) = [x_2, g_2] \in Z(G_2) \backslash \{1\}$. Note that $\beta([x_1, g_1] Z(G_1)) = [x_2, g_2] Z(G_2) = Z(G_2)$ holds, so we have $1 \neq [x_1, g_1] \in Z(G_1)$. This implies $|\bar{C}_1| < |C_1|$, which is a contradiction to $C_1 \in \DCLi{1}$. Hence we obtain $\bar{C}_2 \in \IDCLi{2}$. The other implication follows by symmetry. 
	\bigskip
	
	Extending $\beta$ $F$-linearly gives rise to an $F$-algebra isomorphism $\hat{\beta} \colon F\bar{G}_1 \to F\bar{G}_2$. By the above, we have $\hat{\beta}(\SIDCLi{1}) = \SIDCLi{2}$. Now if $\soc(ZFG_1)$ is an ideal of $FG_1$, Lemma~\ref{lemma:correspondencecenter} implies $\Ann_{ZF\bar{G}_1}(\SIDCLi{1}) \subseteq (\bar{G}_1')^+ \cdot F\bar{G}_1.$ Applying the isomorphism $\hat{\beta}$ yields $\Ann_{ZF\bar{G}_2}(\SIDCLi{2}) \subseteq (\bar{G}_2')^+ \cdot F\bar{G}_2.$ By Lemma~\ref{lemma:correspondencecenter}, $\soc(ZFG_2)$ is an ideal in $FG_2$. The other implication follows by symmetry.
\end{proof}

\subsection{Odd characteristic}\label{sec:oddcharacteristic}
In this section, we assume that $F$ is an algebraically closed field of odd characteristic $p$. 

\begin{Remark}\label{rem:prodelementaryabeliantrivial}
For an abelian $p$-group $G$, we have $\prod_{g \in G} g = 1$ since every nontrivial element in $G$ differs from its inverse and their product is the identity element.
\end{Remark}

\begin{Proposition}\label{prop:npcnonconstantsubgroups}
	Let $G$ be a finite $p$-group of nilpotency class exactly two. Then there exists an element $y \in ZFG$ with $y \notin (G')^+ \cdot FG$ such that $y \cdot S^+ = 0$ holds for all subgroups $1 \neq S \subseteq G'$. 
\end{Proposition}

\begin{proof}
	Since $G'$ is a nontrivial $p$-group, there exists a nontrivial group homomorphism $\alpha \colon G' \to F$. We define an element $y \coloneqq \sum_{g \in G} a_g g \in FG$ by setting $a_g \coloneqq \alpha(g)$ for $g \in G'$ and $a_g = 0$ otherwise. We have $y \in FG' \subseteq FZ(G) \subseteq ZFG$. Now consider a subgroup $1 \neq S \subseteq G'$. The coefficient of $w \in G$ in the product $y \cdot S^+$ is given by $\sum_{s \in S} a_{ws^{-1}}$.
	For $w \notin G'$, all summands are zero. For $w \in G'$, we obtain
	$$\sum_{s \in S} a_{ws^{-1}} = \sum_{s \in S} \alpha(w s^{-1}) = |S| \cdot \alpha(w) + \sum_{s \in S} \alpha(s^{-1})  = \alpha\Bigl(\prod_{s \in S} s^{-1}\Bigr) = \alpha(1) = 0.$$
	In the second and third step, we use that $\alpha$ is a group homomorphism. The fourth equality is due to Remark~\ref{rem:prodelementaryabeliantrivial}. This implies $y \cdot S^+ = 0$ as claimed.
\end{proof}

In this special situation, the condition given in Theorem~\ref{theo:pabeliansocideal} is in fact equivalent to $\soc(ZFG) \trianglelefteq FG$:

\begin{theorem}\label{theo:odd}
	Let $G$ be a finite $p$-group. Then $\soc(ZFG)$ is an ideal in $FG$ if and only if $G$ has nilpotency class at most two. 
\end{theorem}

\begin{proof}
If $G$ is of nilpotency class at most two, we have $G' \subseteq Z(G)$ and hence $\soc(ZFG)$ is an ideal in $FG$ by Theorem~\ref{theo:pabeliansocideal}.
For the converse implication, we use induction on the nilpotency class of $G$. Note that $\bar{G} \coloneqq G/Z(G)$ has nilpotency class $c(G) -1$. First assume $c(G) = 3$. We apply Proposition~\ref{prop:npcnonconstantsubgroups} to the group~$\bar{G}$ and consider the element $y \in ZF\bar{G}$ constructed therein. Let $\bar{C} \in \IDCLab$ be a conjugacy class and let $c \in \bar{C}$. Since $\bar{G}' \subseteq Z(\bar{G})$ holds, the map $\gamma \colon \bar{G} \to \bar{G},\ g \mapsto [g,c]$ is a group homomorphism and hence we have $$\bar{C} = \left\{gcg^{-1} \colon g \in \bar{G}\right\} = \left\{[g,c] c \colon g \in \bar{G}\right\} = S c,$$ where $S \coloneqq \Im \gamma$ is a subgroup of $\bar{G}'$. Note that we have $|S| = |\bar{C}| > 1$. By Proposition \ref{prop:npcnonconstantsubgroups}, we have $y \cdot S^+ = 0$ and hence $y \cdot \bar{C}^+ = y \cdot (Sc)^+ = 0$. Since $y \notin (\bar{G}')^+ \cdot F\bar{G}$ holds, $\soc(ZFG)$ is not an ideal of $FG$ (see Lemma~\ref{lemma:correspondencecenter}). 
	If $G$ is of nilpotency class $c(G) > 3$, we obtain $\soc(ZF\bar{G}) \not \trianglelefteq F\bar{G}$ by induction. Corollary~\ref{cor:socquotient} then yields $\soc(ZFG) \not \trianglelefteq FG$. 
\end{proof}


\begin{Remark}
The analogous construction fails for $p = 2$ since the statement of Remark \ref{rem:prodelementaryabeliantrivial} does not hold for groups of even order.
\end{Remark}

\subsection{\texorpdfstring{Characteristic $p = 2$}{Characteristic p = 2}} \label{sec:evencharacteristic}
Throughout, let $F$ be an algebraically closed field of characteristic two. Unless otherwise stated, we assume that $G$ is a finite $2$-group.

\begin{Remark}\label{rem:fg}
Let $C = \{f,g\}$ be a conjugacy class of length two of $G$. An inner automorphism of~$G$ either fixes both $f$ and $g$, or it interchanges the two elements. For $c \coloneqq g f^{-1} \in G'$, this yields $C_G(f) = C_G(g) \subseteq C_G(c).$ For $h \in G \setminus C_G(f)$, we have $hch^{-1} = hgf^{-1}h^{-1} = fg^{-1} = c^{-1}.$ This shows that the subgroup $\langle c \rangle \subseteq G'$ is normal in $G$. 
\end{Remark}

For every conjugacy class $C \coloneqq \{f,g\}$ of length two, we set $Y_C \coloneqq \langle g f^{-1}\rangle.$ In the following, we consider the subgroup

\begin{equation*}
	Y(G) \coloneqq \langle Y_C: C \in \Cl(G),\ |C| = 2 \rangle.
\end{equation*} 
Note that $Y(G)$ is characteristic in $G$. More precisely, we obtain the following:

\begin{lemma}\label{lemma:ygabelian}
We have $Y(G) \subseteq Z(\Phi(G))$. In particular, $Y(G)$ is abelian. 
\end{lemma}

\begin{proof}
	Note that $Y(G) \subseteq G' \subseteq \Phi(G)$ holds. Now let $C = \{f,g\}$ be a conjugacy class of length two. Since $C_G(f)$ is a maximal subgroup of $G$, Remark \ref{rem:fg} yields $\Phi(G) \subseteq C_G(f) \subseteq C_G(gf^{-1})$ and hence $\Phi(G)$ centralizes~$Y_C$. Thus $Y(G)$ is contained in the center of $\Phi(G)$, so in particular, it is abelian.	
\end{proof}

\begin{lemma}\label{lemma:sociny}
	We have $\soc(ZFG) \subseteq Y(G)^+ \cdot FG.$ 
\end{lemma}

\begin{proof}
	Let $y = \sum_{g \in G} a_g g \in \soc(ZFG)$. For a conjugacy class $C = \{f,g\}$ of length two, we have $c \coloneqq g f^{-1} \in Y(G)$ and the condition $y \cdot C^+ = 0$ yields $a_x = a_{xc^{-1}}$ for all $x \in G$. By
	induction, this implies $a_x = a_{xc_1^{-1} \cdots c_n^{-1}}$ for every $x \in G$ and all elements $c_1, \ldots, c_n$ arising from $G$-conjugacy classes of length two as above. This shows that $y$ has constant coefficients on the cosets of $Y(G)$, that is, we obtain $y \in Y(G)^+ \cdot FG$.
\end{proof}

With this preliminary result, we obtain the following characterization of the $2$-groups $G$ for which $\soc(ZFG)$ is an ideal in $FG$.

\begin{theorem}\label{theo:characterization2groups}
The socle $\soc(ZFG)$ is an ideal in $FG$ if and only if $G' \subseteq Y(G) Z(G)$ holds.
\end{theorem}

\begin{proof}
Suppose $G' \not \subseteq Y(G) Z(G)$, so $Y(G)  Z(G) \cap G'$ is a proper subgroup of $G'$. By \cite[Theorem III.7.2]{HUP67}, there exists a subgroup $N \trianglelefteq G$ with $Y(G) Z(G) \cap G' \subseteq N \subseteq G'$ and $|G' : N| = 2$. We set $M \coloneqq Y(G)Z(G) N$. Note that $M^+ \in ZFG$ holds since $M$ is a normal subgroup of $G$. We now show that $M^+$ annihilates the basis of $J(ZFG)$ given in Theorem~\ref{theo:structjzfg}. 
	\bigskip
	
	For $z \in Z(G) \subseteq M$, we have $(1+z) \cdot M^+ = 0$. For a $G$-conjugacy class $C = \{f,g\}$ of length two, we obtain $C^+ \cdot Y(G)^+ = f Y(G)^+ + g Y(G)^+ = 0$ since $g f^{-1} \in Y(G)$ holds. Hence $M^+$ annihilates $C^+$. Every conjugacy class $C \in \Cl(G)$ with $|C| \geq 4$ contains an even number of elements in every coset of $N$ since $C$ is contained in a coset of $G'$ and $|G' : N| = 2$ holds. This implies that $C^+$ is annihilated by $N^+$ and hence by $M^+$. Summarizing, we obtain $M^+ \in \soc(ZFG)$. Moreover, $M \cap G' = N \subsetneq G'$ implies $M^+ \notin (G')^+ \cdot FG$. By Lemma \ref{lemma:socideal}, this yields $\soc(ZFG) \not\trianglelefteq FG$.
	\bigskip
	
	Conversely, assume that $G' \subseteq Y(G) Z(G)$ holds. By Lemmas~\ref{lemma:centerpgroups} and~\ref{lemma:sociny}, we have $$\soc(ZFG) \subseteq (Y(G) Z(G))^+ \cdot FG \subseteq (G')^+ \cdot FG$$ and hence $\soc(ZFG)$ is an ideal of $FG$ (see Lemma \ref{lemma:socideal}). 
\end{proof}

This completes the proof of Theorem~\ref{theo:pgroups}.

\begin{Remark}
Similarly to the case of odd characteristic, $\soc(ZFG) \trianglelefteq FG$ holds if $G$ has nilpotency class at most two. 
\end{Remark}

However, the next example demonstrates that in contrast to the case of odd characteristic, the nilpotency class of a finite $2$-group $G$ for which $\soc(ZFG)$ is an ideal in $FG$ can be arbitrarily large.

\begin{Example}
$\null$ 
\begin{enumerate}[(i)]
\item Let $G = D_{2^n} = \langle r, s \colon r^{2^{n-1}} = s^2 = 1,\ srs = r^{-1} \rangle$ with $n \in \N$ be the dihedral group of order $2^n$. For $n \leq 2$, $G$ is abelian and hence $\soc(ZFG) \trianglelefteq FG$ holds. For $n \geq 3$, we have $G' = \langle r^2 \rangle = Y(G) Z(G)$ and hence $\soc(ZFG) \trianglelefteq FG$ follows by Theorem \ref{theo:characterization2groups}. The $2$-groups of maximal class of a fixed order are isoclinic. Therefore, by Theorem~\ref{theo:isoclinic}, $\soc(ZFG)$ is an ideal in $FG$ if $G$ is a semihedral or generalized quaternion $2$-group.
\item By \cite[Theorem 4.12]{BRE221}, every $2$-group $G$ of order at most $16$ satisfies $\soc(ZFG) \trianglelefteq FG$. Up to isomorphism, there exist $51$ groups of order $32$. Out of those, $7$ groups are abelian and $26$ groups have nilpotency class precisely two. Additionally, $13$ groups satisfy the property $G' \subseteq Y(G) Z(G)$.
\item Consider the holomorph $G = \Z/8\Z \rtimes (\Z/8\Z)^\times$ of $\Z/8\Z$, which has order $32$. It has $11$ conjugacy classes and we have $|Z(G)| = 2$. Since $G/Z(G) \cong D_8 \times C_2$ has precisely $10$ conjugacy classes, the images of the non-central conjugacy classes of $G$ in $G/Z(G)$ are pairwise distinct. For every such conjugacy class $C$, we therefore have $Z(G) C \subseteq C$ and hence $\nu_{Z(G)}(C^+) = 0$. This proves $J(ZFG)^2 = 0$, so $J(ZFG) = \soc(ZFG)$ follows. In particular, we obtain $\dim \soc(ZFG) = \dim J(ZFG) = 10$. Due to $|G'| = 4$, the space $(G')^+ \cdot FG$ is eight-dimensional, so it does not contain $\soc(ZFG)$. By Lemma~\ref{lemma:socideal}, $\soc(ZFG)$ is not an ideal in $FG$.  
\end{enumerate}
\end{Example}

We conclude this part with a generalization of Theorem~\ref{theo:pabeliansocidealeven} to arbitrary finite groups, which is a stronger variant of Theorem~\ref{theo:pabeliansocideal}: 

\begin{theorem}\label{theo:pabeliansocidealeven}
Let $G$ be an arbitrary finite group which satisfies $G' \subseteq Y(O_2(G)) Z(O_2(G))$. Then $\soc(ZFG)$ is an ideal of $FG$. 
\end{theorem}

\begin{proof}
The given condition implies $G' \subseteq O_2(G)$, so by Theorem~\ref{theo:reynoldsideal}, we have $G = P \rtimes H$ with $P \coloneqq O_2(G) \in \Syl_2(G)$ and an abelian $2'$-group $H$. Note that $G'$ is abelian as $Y(P)$ is abelian (see Lemma~\ref{lemma:ygabelian}). By Remark~\ref{rem:trivianeu}, we have 
		\begin{equation}\label{eq:decompy}
		P = C_P(H) G' = C_P(H) Y(P) Z(P).
		\end{equation}
Since $G'$ is abelian, $C_P(H)'$ is normal in $P$. We consider the group $\bar{P} \coloneqq P/C_P(H)'$ and denote the image of $S \subseteq P$ in $\bar{P}$ by $\bar{S}$. Then we have 
\[
\overline{Y(P)} \subseteq \bar{P}' = \overline{(C_P(H)Y(P))'} = [\overline{C_P(H)}, \overline{Y(P)}] \subseteq [\bar{P}, \overline{Y(P)}].
\]

This implies $\overline{Y(P)} = 1$, so $Y(P) \subseteq C_P(H)'$ follows. 

		\bigskip
		
		By \eqref{eq:decompy}, we then have $[P,H] = [Z(P),H]$ and hence $P =C_P(H) [Z(P),H]$ follows. Since $C_P(H)$ centralizes $W \coloneqq H [Z(P),H]\subseteq H Z(P)$ and $C_P(H) \cap [Z(P),H]  = 1$ follows by \cite[Theorem 5.3.6]{GOR68}, we obtain $G = C_P(H) \times W$. It is then easily verified that $Y(P) = Y(C_P(H))$ holds. 
		With Dedekind's identity, we obtain $$C_P(H)' \subseteq G' \cap C_P(H) \subseteq Y(C_P(H)) Z(P) \cap C_P(H) \subseteq Y(C_P(H)) \cdot Z(C_P(H)).$$ By Theorem \ref{theo:characterization2groups}, $\soc(ZFC_P(H))$ is an ideal of $FC_P(H)$. Since $\soc(ZFW) \trianglelefteq FW$ follows by Theorem~\ref{theo:pabeliansocideal}, $\soc(ZFG)$ is an ideal in $FG$ by Theorem~\ref{theo:centralproduct}.
\end{proof}

This completes the proof of Theorem~\ref{theo:c}.

\section{\texorpdfstring{Decomposition of $G$ into a central product}{Decomposition of G into a central product}}\label{sec:decomp}

Let $F$ be an algebraically closed field of characteristic $p > 0$. 
We consider an arbitrary finite group $G$ for which $\soc(ZFG)$ is an ideal in $FG$. 
By Theorem~\ref{theo:reynoldsideal}, we may write $G = P \rtimes H$ with $P \in \Syl_p(G)$ and an abelian $p'$-group $H$. 
In this section, we prove Theorem~\ref{theo:decompositionofp}. Combined with the results on $p$-groups from the last section, it reduces our investigation to the case that $G'$ is a Sylow $p$-subgroup of $G$.

\setcounter{theoremintro}{3}
\begin{theoremintro}
We have $G = C_P(H) * O^p(G)$. Moreover, $\soc(ZFC_P(H))$ and $\soc(ZFO^p(G))$ are ideals in $FC_P(H)$ and $FO^p(G)$, respectively. The socle of $ZFG$ is explicitly given by 
	\[\soc(ZFG) = (Z(P)G')^+ \cdot FG.\]
\end{theoremintro}

\begin{proof}
By Proposition~\ref{prop:nilpotencyclassDG}, we have $G = C_P(H) * O^p(G)$. Theorem~\ref{theo:centralproduct} then implies that $\soc(ZFC_P(H))$ and $\soc(ZFO^p(G))$ are ideals in $FC_P(H)$ and $FO^p(G)$, respectively.
It therefore remains to determine the structure of $\soc(ZFG)$. By the above decomposition, we have $Z(C_P(H)) \subseteq Z(G)$. By Corollary~\ref{cor:zpdginsoc}, we obtain $\soc(ZFG) \subseteq Z(C_P(H))^+ \cdot FG$. 
	Together with Lemma~\ref{lemma:socideal}, this implies $$\soc(ZFG) \subseteq (Z(C_P(H))G')^+ \cdot FG \subseteq (Z(P)G')^+ \cdot FG.$$ In the last step, we used $Z(P) = Z(C_P(H)) Z([G',H]) \subseteq Z(C_P(H)) G'$. On the other hand, we have $(Z(P)G')^+ \cdot FG \subseteq \soc(ZFG)$ by Lemma~\ref{lemma:312general}, which completes the proof.
\end{proof}

This result on the structure of $\soc(ZFG)$ generalizes the corresponding statement in Lemma~\ref{lemma:centerpgroups}. Note that, by Theorem~\ref{theo:decompositionofp}, the hypothesis that $\soc(ZFG)$ is an ideal in 
$FG$ implies $\dim \soc(ZFG) = |G:G'Z(G)|$. In particular, the dimension of $\soc(ZFG)$ is a divisor of $|G|$. We also observe that in this situation, $\soc(ZFG)$ is a principal ideal of $FG$ generated by a central element. Furthermore, we obtain the following reduction:

\begin{Remark}
Since the structure of the $p$-group $C_P(H)$ is determined by Theorem~\ref{theo:pgroups}, it suffices to investigate the group $O^p(G)$. Inductively, we may assume $O^p(G) =G$. By Remark~\ref{rem:trivianeu}, this yields $P = G' = [G',H]$. In particular, $C_{G'}(H) \subseteq G'' \subseteq Z(G')$ follows (see \cite[Theorem 5.2.3]{GOR68}), which implies $C_{G'}(H) \subseteq Z(G)$. If additionally $O_{p'}(G) = 1$ holds, we obtain $C_{G'}(H) = Z(G)$. 
\end{Remark}

Moreover, we state the following consequence of Theorem~\ref{theo:decompositionofp}: 

\begin{theorem}
	We have $\soc(ZFP) \trianglelefteq FP$. In particular, the group $P$ is metabelian and its nilpotency class is at most two if $p$ is odd.
\end{theorem}

\begin{proof}
	By Theorem~\ref{theo:decompositionofp}, we have $P = C_P(H) * [P,H]$ and $\soc(ZFC_P(H))$ is an ideal in $FC_P(H)$. Since $[P,H] \subseteq G'$ has nilpotency class at most two (see Proposition~\ref{prop:nilpotencyclassDG}), we obtain $\soc(ZF[P,H]) \trianglelefteq F[P,H]$ by Theorem~\ref{theo:pgroups}. By Theorem~\ref{theo:centralproduct}, this yields $\soc(ZFP) \trianglelefteq FP$. In particular, it follows that $P$ is metabelian and that the nilpotency class of $P$ is at most two if $p$ is odd (see Theorem~\ref{theo:pgroups}).
\end{proof}

\section*{Acknowledgments}
The results of this paper form a part of the PhD thesis of the first author \cite{BRE22}, which was supervised by the second author.

\bibliographystyle{plain}
\bibliography{bibpgroups.bib}
\end{document}